\documentclass[12pt]{amsart}
\usepackage{amssymb,amsmath,amscd,latexsym,epsfig,color,hyperref,graphics}
\usepackage[all,cmtip]{xy}

\usepackage[margin=1in]{geometry}

\newtheorem{theorem}{Theorem}[section]
\newtheorem{proposition}[theorem]{Proposition}
\newtheorem{corollary}[theorem]{Corollary}

\newtheorem{lemma}[theorem]{Lemma}

\theoremstyle{definition}
\newtheorem{definition}[theorem]{Definition}
\newtheorem{remark}[theorem]{Remark}
\newtheorem{example}[theorem]{Example}

\def\NN{\ensuremath{\mathbb{N}}}

\def\RR{\ensuremath{\mathbb{R}}}

\newcommand{\V}{{\mathcal V}}

\def\K{\ensuremath{\mathcal{K}}}

\def\C{\ensuremath{\mathcal{C}}}

\def\rank{\ensuremath{\textup{rank}}}

\def\min{\ensuremath{\textup{min}}}
\def\supp{\ensuremath{\textup{supp}}}

\def\conv{\ensuremath{\textup{conv}}}

\def\ext{\ensuremath{\textup{ext}}}
\def\ker{\ensuremath{\textup{ker}}}

\def\PSD{\ensuremath{\mathcal{S}_{+}}}

\def\STAB{\ensuremath{\textup{STAB}}}
\def\Aut{\ensuremath{\textup{Aut}}}

\def\GL{\ensuremath{\textup{GL}}}

\def\diag{\ensuremath{\textup{diag}}}
\def\TH{\ensuremath{\textup{TH}}}

\title{Lifts of Convex Sets and Cone Factorizations} 
  
\author{Jo{\~a}o Gouveia}
\address{CMUC, Department of Mathematics,
  University of Coimbra, 3001-454 Coimbra, Portugal}
\email{jgouveia@mat.uc.pt} 

\author{Pablo A. Parrilo}
\address{Department of Electrical Engineering and Computer Science,
  Laboratory for Information and Decision Systems, Massachusetts
  Institute of Technology, 77 Massachusetts Avenue, Cambridge, MA
  02139-4307, USA} 
\email{parrilo@mit.edu} 

\author{Rekha Thomas}
\address{Department of Mathematics, University of Washington, Box
  354350, Seattle, WA 98195, USA} \email{thomas@math.washington.edu}

\thanks{All authors were partially supported by grants from the U.S. National Science Foundation.
Gouveia was also supported by Funda{\c c}{\~ a}o para a Ci{\^ e}ncia e Tecnologia.}  

\date{\today}
\begin{document}

\begin{abstract}
In this paper we address the basic geometric question of when a given convex set is the image under a linear map of an affine slice of a given closed convex cone. Such a  representation or ``lift'' of the convex set is especially useful if the cone admits an efficient algorithm for linear optimization over its affine slices.  
We show that the existence of a lift of a convex set to a cone is equivalent to the existence of a factorization 
of an operator associated to the set and its polar via elements in the cone and its dual. This generalizes a theorem of Yannakakis that established a connection between polyhedral lifts of a polytope and nonnegative factorizations of its slack matrix. Symmetric lifts of convex sets can also be characterized similarly.
When the cones live in a family, our results lead to the definition of the rank of a convex set with respect to this family. We present results about this rank in the context of cones of positive semidefinite matrices. Our methods provide new tools for understanding 
cone lifts of convex sets.
\end{abstract}

\maketitle

\section{Introduction} 
\label{sec:intro}

Linear optimization over convex sets plays a central role in
optimization. In many instances, a convex set $C \subset \RR^n$ may
come with a complicated representation that cannot be altered if one
is restricted in the number of variables and type of representation
that can be used. For instance, the $n$-dimensional cross-polytope
$$C_n := \{ x \in \RR^n \,:\, \pm x_1 \pm x_2 \cdots \pm x_n \leq 1
\}$$ requires the above $2^n$ constraints in any representation of it
by linear inequalities in $n$ variables.  However, $C_n$ is the
projection onto the $x$-coordinates of the polytope
$$ Q_n := \{ (x,y) \in \RR^{2n} \,:\, \sum_{i=1}^{n} y_i = 1, -y_i
\leq x_i \leq y_i \,\forall \, i=1, \ldots, n \}$$ which is described
by $2n+1$ linear constraints and $2n$ variables, and one can optimize
a linear function $\langle c, x \rangle$ over $C_n$ by instead
optimizing it over $Q_n$. Since the running time of linear programming
algorithms depends on the number of linear constraints of the feasible
region, the latter representation allows rapid optimization over
$C_n$. More generally, if a convex set $C \subset \RR^n$ can be
written as the image under a linear map of an affine slice of a cone
that admits efficient algorithms for linear optimization, then one can
optimize a linear function efficiently over $C$ as well. For instance,
linear optimization over affine slices of the $k$-dimensional
nonnegative orthant $\RR^k_+$ is linear programming, and over the cone
of $k \times k$ real symmetric positive semidefinite matrices
$\mathcal{S}^k_+ $ is semidefinite programming, both of which admit
efficient algorithms. Motivated by this fact, we ask the following
basic geometric questions about a given convex set $C \subset \RR^n$:

\begin{center}
\begin{enumerate}
  \item {\em Given a full-dimensional closed convex cone $K \subset
  \RR^m$, when does there exist an affine subspace $L \subset \RR^m$
 and a linear map $\pi \,:\, \RR^m \rightarrow \RR^n$ such that $C =
                          \pi (K \cap L)$?}
\item {\em If the cone $K$ comes from a family $( K_k )$ (e.g. $( \RR^k_+ )$ or $( \PSD^k )$), then what is the least $k$ for which $C = \pi(K_k \cap L)$ for some $\pi$ and $L$?}
\end{enumerate}
\end{center}

If $C = \pi(K \cap L)$, then $K \cap L$ is called a $K$-{\em
  lift} of $C$. In \cite{Yannakakis}, Yannakakis points out a remarkable
connection between the smallest $k$ for which a polytope has a $\RR^k_+$-lift 
and the {\em nonnegative rank} of its {\em slack matrix}.
The main result of our paper is an extension of
Yannakakis' result to the general scenario of $K$ being any closed
convex cone and $C$ any convex set, answering Question (1) above. The
main tool is a generalization of nonnegative factorizations of
nonnegative matrices to {\em cone factorizations} of {\em slack
  operators} of convex sets.

This paper is organized as follows. In Section~\ref{sec:mainthms} we
present our main result (Theorem~\ref{thm:general_Yannakakis})
characterizing the existence of a $K$-lift of a convex set $C \subset
\RR^n$, when $K$ is a full-dimensional closed convex cone in
$\RR^m$. A $K$-lift of $C$ is {\em symmetric} if it respects the
symmetries of $C$. In Theorem~\ref{thm:general_Yannakakis_symm}, we
characterize the existence of a {\em symmetric} $K$-{\em lift} of $C$. Although
symmetric lifts are quite special, they have received much
attention. The main result in \cite{Yannakakis} was that a symmetric
$\RR^k_+$-lift of the {\em matching polytope} of the complete graph on
$n$ vertices requires $k$ to be at least exponential in
$n$. Results in \cite{KaibelPashkovichTheis} and \cite{Pashkovich} 
have shown that symmetry imposes strong
restrictions on the minimum size of polyhedral
lifts. Proposition~\ref{prop:geometric operations} describes geometric
operations on convex sets that preserve the existence of cone lifts.

In Section~\ref{sec:polytopes} we focus on polytopes. As a corollary of
Theorem~\ref{thm:general_Yannakakis} we obtain
Theorem~\ref{thm:general_Yannakakis_for_polytopes} which generalizes
Yannakakis' result for polytopes \cite[Theorem~3]{Yannakakis} to
arbitrary closed convex cones $K$. We illustrate
Theorems~\ref{thm:general_Yannakakis_for_polytopes} and
\ref{thm:general_Yannakakis_symm} using polygons in the plane.

Section~\ref{sec:coneranks} tackles Question (2) and considers ordered
families of cones, $\K= (K_k)$, that can be used to lift a given $C
\subset \RR^n$, or more simply, to factorize a nonnegative matrix $M$.
When all faces of all cones in $\K$ are again in $\K$, we define
$\rank_{\K}(C)$ (respectively, $\rank_{\K}(M)$) to be the smallest $k$
such that $C$ has a $K_k$-lift (respectively, $M$ has a
$K_k$-factorization).  We focus on the case of $\K = ( \RR^k_+ )$ when
$\rank_{\K}(\cdot)$ is called nonnegative rank, and $\K = (\PSD^k)$
when $\rank_{\K}(\cdot)$ is called psd rank. Section~\ref{subsec:cone
  ranks defs} gives the basic definitions and properties of cone
ranks.  We find (different) families of nonnegative matrices that show
that the gap between any pair among: rank, psd rank and nonnegative
rank, can become arbitrarily large.  In Section~\ref{subsec:cone ranks
  of convex sets} we derive lower bounds on nonnegative and psd ranks
of polytopes. We note that the nonnegative rank of a polytope is
  also called the {\em extension complexity} of the polytope by some
  authors in reference to this invariant being the smallest $k$ for
  which the polytope admits a $\RR^k_+$-lift.
Corollary~\ref{cor:bounds on nonneg rank examples} shows a lower bound
for the nonnegative rank of a polytope in terms of the size of an
antichain of its face lattice. Corollary~\ref{cor:max facets with
  fixed psd rank} gives an upper bound on the number of facets of a
polytope with psd rank $k$. This subsection also finds families of
polytopes whose slack matrices exhibit arbitrarily large gaps between
rank and nonnegative rank, as well as rank and psd rank.

In Section~\ref{sec:applications} we give two applications of our
methods.  When $C = \STAB(G)$ is the {\em stable set polytope} of a
graph $G$ with $n$ vertices, Lov{\'a}sz constructed a convex
approximation of $C$ called the {\em theta body} of $G$. This body is
the projection of an affine slice of $\PSD^{n+1}$, and when $G$ is a
{\em perfect graph}, it coincides with $\STAB(G)$. Our methods show that 
this construction is optimal in the sense that for any $G$, $\STAB(G)$ cannot admit a
$\PSD^k$-lift for any $k \leq n$. A result of Burer
shows that every $\STAB(G)$ has a $\mathcal C^*_{n+1}$-lift where
$\mathcal C^*_{n+1}$ is the cone of {\em completely positive matrices} of
size $(n+1) \times (n+1)$. We illustrate Burer's result in terms of
Theorem~\ref{thm:general_Yannakakis} on a cycle of length five.  The
second part of Section~\ref{sec:applications} interprets
Theorem~\ref{thm:general_Yannakakis} in the context of {\em rational 
  lifts} of convex hulls of algebraic sets. We show in
Theorem~\ref{thm:moments} that in this case, the positive semidefinite factorizations
required by Theorem~\ref{thm:general_Yannakakis} can be interpreted in
terms of sums of squares polynomials and rational maps.

In the last few decades, several {\em lift-and-project} methods have
been proposed in the optimization literature that aim to provide
tractable descriptions of convex sets.  These
methods construct a series of nested convex approximations to $C \subset \RR^n$ that
arise as projections of higher dimensional convex sets. Examples 
can be found in \cite{Balas79, SheraliAdams90,LovaszSchrijver91, Lasserre1, Parrilo:spr, GPT1, KojimaTuncel} 
and \cite{BienstockZuckerberg}. In these methods, $C$ is either a $0/1$-polytope or more generally, the convex hull of a  semialgebraic set, and the cones that are used in the lifts are either nonnegative orthants or the cones of positive semidefinite matrices.  The success of a lift-and-project method relies on whether 
a lift of $C$ is obtained at some step of the procedure. 
Questions (1) and (2), and our answers to them, address this convergence question and offer a uniform framework within which to study all lift-and-project methods for convex sets using closed convex cones.

There have been several recent developments that were motivated by the results of Yannakakis in 
\cite{Yannakakis}. As mentioned earlier, Kaibel, Pashkovich and Theis proved that symmetry can impose severe 
restrictions on the minimum size of a polyhedral lift of a polytope. An exciting new result of Fiorini, Massar, Pokutta, Tiwary and de Wolf shows that there are {\em cut}, {\em stable set} and {\em traveling salesman} polytopes for which there can be no polyhedral lift of size polynomial in the number of vertices of the associated graphs. Their paper \cite{FMPTW} also gives an interpretation of positive semidefinite rank of a nonnegative matrix in terms of quantum communication complexity extending the connection between nonnegative rank and classical communication complexity established in \cite{Yannakakis}.

\section{Cone lifts of convex bodies} \label{sec:mainthms}

A convex set is called a {\em convex body} if it is compact and contains the 
origin in its interior. To simplify notation, we will assume throughout the paper that the
convex sets $C \subset \RR^n$ for which we wish to study cone lifts are 
all convex bodies, even though our results hold for all convex sets.
Recall that the {\em polar} of a convex set $C \subset \RR^n$ is the set
\[
C^\circ=\{y \in \RR^n: \langle x,y \rangle \leq 1, \ \ \forall x \in C \}.
\] 
Let $\ext(C)$ denote the set of {\em extreme points} of
$C$, namely, all points $p \in C$ such that if $p = (p_1 + p_2)/2$, with $p_1,p_2 \in C$, then 
$p=p_1=p_2$.  Since $C$ is compact with the origin in its interior, both $C$ and $C^\circ$ are convex hulls of their  
respective extreme points. Consider the operator $S:\RR^n\times \RR^n
\rightarrow \RR$ defined by $S(x,y)=1-\left<x,y\right>$. We define the {\em slack operator} $S_C$, of the convex set $C$, 
to be the restriction of $S$ to $\ext(C) \times \ext(C^\circ)$. 

\begin{definition} \label{def:K-lift}
Let $K \subset \RR^m$ be a full-dimensional closed convex cone and $C \subset \RR^n$ a full-dimensional 
convex body. A {\em $K$-lift} of $C$ is a set $Q=K \cap L$,
where $L \subset \RR^m$ is an affine subspace, and $\pi:\RR^m \rightarrow \RR^n$ is a linear map 
such that $C=\pi(Q)$. If
$L$ intersects the interior of $K$ we say that $Q$ is a {\em
  proper $K$-lift} of $C$.
\end{definition}

We will see that the existence of a $K$-lift of $C$ is intimately connected
to properties of the slack operator $S_C$. Recall that the {\em dual} of a 
closed convex cone $K \subset \RR^m$ is 
$$K^*=\{y \in \RR^m: \langle x,y \rangle \geq 0, \ \ \forall x \in K
\}.$$ A cone $K$ is \emph{self-dual} if $K^* = K$. In particular,
  the cones $\RR_+^n$ and $\PSD^k$ are self-dual.

\begin{definition}
Let $C$ and $K$ be as in Definition~\ref{def:K-lift}. We say that the slack operator 
$S_C$ is
\emph{$K$-factorizable} if there exist maps (not necessarily linear)
$$A:\ext(C)\rightarrow K \,\,\,\,\textup{and} \,\,\,\, B:\ext(C^\circ)\rightarrow K^{*}$$
such
that $S_C(x,y)=\left<A(x),B(y)\right>$ for all $(x,y) \in \ext(C)
\times \ext(C^\circ)$.
\end{definition}

\begin{remark}\label{remark:extension} The maps $A$ and $B$ may be defined over all of $C$ and $C^\circ$ 
by picking a representation of each $x \in C$ (similarly, $y \in C^\circ$) as a convex combination of extreme points of $C$ 
(respectively, $C^\circ$) and extending $A$ and $B$ linearly. Such extensions are not unique.
\end{remark}

With the above set up, we can now characterize 
the existence of a $K$-lift of $C$.

\begin{theorem} 
\label{thm:general_Yannakakis}
If $C$ has a proper $K$-lift then $S_C$ is
$K$-factorizable. Conversely, if $S_C$ is $K$-factorizable then $C$
has a $K$-lift.
\end{theorem}

\begin{proof}
Suppose $C$ has a proper $K$-lift. Then there exists an affine space $L = w_0+L_0$ in $\RR^m$ ($L_0$ is a linear subspace) and a linear map $\pi:\RR^m
\rightarrow \RR^n$ such that $C = \pi(K \cap L)$ and $w_0 \in
\textup{int}(K)$. Equivalently, 
\[
C = \{ x \in \RR^n : x = \pi(w), \quad w \in K \cap (w_0 + L_0) \}.
\]
We need to construct the maps $A \,:\, \textup{ext}(C) \rightarrow K$ and 
$B:\textup{ext}(C^\circ)\rightarrow K^*$ that factorize the
slack operator $S_C$, from the $K$-lift of $C$. 
For $x_i \in \textup{ext}(C)$, define $A(x_i) :=
w_i$, where $w_i$ is any point in the non-empty convex set 
$\pi^{-1} (x_i) \cap K \cap L$.

Let $c$ be an extreme point of $C^\circ$. Then $\textup{max} \{ \,
\langle c, x \rangle \,:\, x \in C \, \} = 1$ since $\langle c, x
\rangle \leq 1$ for all $x \in C$, and if the maximum was smaller than
one, then $c$ would not be an extreme point of $C^{\circ}$.  Let $M$
be a full row rank matrix such that $\textup{ker} \,M = L_0$. Then the
following hold:
$$
\begin{array}{cccccc}
\begin{array}{c} 1 = \\ \\ \\ \end{array} &
\begin{array}{c} \textup{max} \langle c,x \rangle \\  x \in C \\ \\ \end{array} & 
\begin{array}{c} = \\ \\ \\ \end{array} &
\begin{array}{c} \textup{max} \langle c, \pi(w) \rangle \\ w  \in K \cap (w_0 + L_0) \\ \\ \end{array} & 
\begin{array}{c} = \\ \\ \\ \end{array} &
\begin{array}{c} \textup{max} \langle \pi^*(c), w \rangle \\ M w  = M w_0 \\ w \in K \end{array} 
\end{array}
$$
Since $w_0$ lies in the interior of $K$,  by Slater's condition we have strong duality, and we get
$$ 1 = \textup{min} \,\langle M w_0, y \rangle \,:\, M^T y - \pi^*(c) \in K^*$$
with the minimum being attained.
Further, setting $z = M^T y$ we have that 
$$ 1 = \textup{min} \, \langle w_0, z  \rangle \,:\,   z - \pi^*(c) \in K^*, \, z \in L_0^\perp $$
with the minimum being attained. Now define 
$B \,:\, \textup{ext}(C^\circ) \rightarrow K^*$ as the map that sends $y_i \in \textup{ext}(C^\circ)$ to 
$B(y_i) := z - \pi^*(y_i)$, where $z$ is any point in the nonempty convex set $L_0^\perp
\cap (K^* + \pi^*(y_i))$ that satisfies $\langle w_0,z \rangle =1$. Note that for such a $z$, $\langle w_i, z \rangle = 1$ for all $w_i \in L$.
Then $B(y_i) \in K^*$, and for an $x_i \in \textup{ext}(C)$,
\begin{align*}
\langle x_i, y_i \rangle &=
\langle \pi (w_i), y_i \rangle  = 
\langle w_i, \pi^* (y_i )\rangle =
\langle w_i , z - B(y_i) \rangle  \\ 
&= 1 - \langle w_i, B(y_i) \rangle =
1 - \langle A(x_i), B(y_i) \rangle.
\end{align*}

Therefore, $S_C(x_i,y_i) = 1 - \langle x_i, y_i \rangle = \langle A(x_i), B(y_i) \rangle$ for all $x_i \in \textup{ext}(C)$ and 
$y_i \in \textup{ext}(C^\circ)$. 

Suppose now $S_C$ is $K$-factorizable, i.e., there exist maps
$A:\ext(C)\rightarrow K$ and $B:\ext(C^\circ)\rightarrow K^*$ such
that $S_C(x,y)=\left<A(x),B(y)\right>$ for all $(x,y) \in \ext(C)
\times \ext(C^\circ)$. Consider the affine space
\[
L=\{(x,z) \in \RR^n \times \RR^m : 1 - \langle x, y \rangle =\left<z,B(y)\right>,
\ \forall \,\,y \in \ext(C^\circ)\},
\] 
and let $L_K$ be its coordinate projection into $\RR^m$. Note that $0 \not \in L_K$ since otherwise, 
there exists $x \in \RR^n$ such that $1- \langle x, y \rangle = 0$ for all $y \in \ext(C^\circ)$ 
which implies that $C^\circ$ lies in the affine hyperplane $\langle x, y \rangle = 1$. This is a contradiction 
since $C^\circ$ contains the origin. Also, $K \cap L_K \neq \emptyset$ since for each $x \in \ext(C)$, 
$A(x) \in K \cap L_K$ by assumption.

Let $x$ be some point in $\RR^n$ such that there exists some $z \in K$
for which $(x,z)$ is in $L$. Then, for all extreme points $y$ of
$C^\circ$ we will have that $1-\langle x,y \rangle$ is nonnegative. This implies,
using convexity, that $1-\left< x ,y \right>$ is nonnegative for all
$y$ in $C^\circ$, hence $x \in (C^\circ)^\circ=C$.

We now argue that this implies that for each $z \in K \cap L_K$ there exists a unique $x_z
\in \RR^n$ such that $(x_z,z) \in L$. That there is one, comes
immediately from the definition of $L_K$. Suppose now that there is
another such point $x_{z}'$. Then $(t x_{z}+(1-t)x_{z}',z) \in L$ for
all reals $t$ which would imply that the line through $x_{z}$ and
$x_{z}'$ would be contained in $C$, contradicting our assumption that $C$ is 
compact.

The map that sends $z$ to $x_{z}$ is therefore well-defined in $K \cap
L_K$, and can be easily checked to be affine. Since the origin is not in $L_K$, we can extend it to
a linear map $\pi:\RR^m \rightarrow \RR^n$.  To finish the proof it
is enough to show $ C = \pi(K \cap L_K)$. We have already seen that
$\pi(K \cap L_K) \subseteq C$ so we just have to show the reverse
inclusion. For all extreme points $x$ of $C$, $A(x)$ belongs to $K \cap L_K$, and 
therefore, $x=\pi(A(x)) \in \pi(K \cap L_K)$. Since $C =
\conv(\ext(C))$ and $\pi(K \cap L_K)$ is convex, $C \subseteq \pi(K \cap L_K)$.
\end{proof}

The restriction to proper lifts in 
  Theorem~\ref{thm:general_Yannakakis} is not important if the
  cone $K$ has a well-understood facial structure as in the case of
  nonnegative orthants and cones of positive semidefinite matrices.
  If there exists a $K$-lift that is not proper, then there is a
  proper lift to a face of $K$ and we could pass to this face to
  obtain a cone factorization. Since our proof uses strong
  duality, it is not obvious how to remove the properness assumption
  for a general closed convex cone. However, there is a situation
  under which properness can be dropped.

\begin{definition} \label{def:nice cone} \cite{BorweinWolkowicz}
A cone $K$ is {\em nice} if $K^* + F^{\perp}$ is closed for all faces $F$ of $K$. 
\end{definition} 

\begin{corollary} If $K$ is a nice cone, then whenever $C$ has a $K$-lift (not necessarily proper), $S_C$ has a 
$K$-factorization. 
\end{corollary}

\begin{proof}
In \cite{Pataki2011}  Pataki notes that $K$ is nice if and only if 
$F^*=K^*+F^{\perp}$ for all faces $F$ of $K$. Let 
$A \, : \, \textup{ext}(C) \rightarrow  F$ and $B \, : \, \textup{ext}(C^\circ) \rightarrow F^*$ be the 
$F$-factorization of $S_C$ from the proper lift of $C$ to a face $F$ of $K$. Then $A$ is also a 
map from $\textup{ext}(C)$ to $K$. Define $B' \,:\, \ext(C^\circ) \rightarrow K^*$ as 
$B'(y) = z \in K^*$ such that 
$B(y) - z \in F^\perp$. Then $\langle A(x), B(y) \rangle = \langle A(x), B'(y) \rangle$ for all 
$(x,y) \in \textup{ext}(C) \times \textup{ext}(C^\circ)$ and we obtain a $K$-factorization of $S_C$.
\end{proof}

Polyhedral cones, second order cones and the cones of real symmetric psd matrices $\mathcal{S}^k_+$ 
are all nice. In \cite{Pataki2011} Pataki shows that 
if a cone is nice then all its faces are exposed and he conjectures that the converse is also true.

We now present a simple illustration
of Theorem~\ref{thm:general_Yannakakis} using $K = \PSD^2$.

\begin{example}
Let $C$ be the unit disk in $\RR^2$ which can be written as
\[
C=\left\{
(x,y) \in \RR^2 : \left( \begin{array}{cc}
1+x & y \\
y & 1-x
\end{array}\right) \succeq 0
\right\}.
\]
This means that $S_C$ must have a $\PSD^2$ factorization. Since $C^\circ=C$, $\textup{ext}(C) = \textup{ext}(C^\circ) = \partial C$, and so we have to find maps $A,B: \textup{ext}(C) \rightarrow \PSD^2$ such
that for all $(x_1,y_1), (x_2,y_2) \in \ext(C)$,
\[
\left<A(x_1,y_1),B(x_2,y_2)\right> = 1 -x_1x_2 - y_1y_2.
\] 
But this is accomplished by the maps
\[
A(x_1,y_1)=\left( \begin{array}{cc}
1+x_1 & y_1 \\
y_1 & 1-x_1
\end{array}\right)
\]
and
\[
B(x_2,y_2)=\frac{1}{2}\left( \begin{array}{cc}
1-x_2 & -y_2 \\
-y_2 & 1+x_2
\end{array}\right)
\]
which factorize $S_C$ and can easily be checked to be positive semidefinite in their domains.
\end{example}

The lifts of convex bodies are preserved by many common geometric operators. 

\begin{proposition} \label{prop:geometric operations}
If $C_1$ and $C_2$ are convex bodies, and $K_1$ and $K_2$ are closed convex cones such that $C_1$ has a $K_1$-lift and $C_2$ has a $K_2$-lift, then the following are true:
\begin{enumerate}
\item If $\pi$ is any linear map, then $\pi(C_1)$ has a $K_1$-lift;
\item $C_1^{\circ}$ has a $K_1^*$-lift;
\item Every exposed face of $C_1$ has a $K_1$-lift;
\item The cartesian product $C_1 \times C_2$ has a $K_1 \times K_2$-lift;
\item The Minkowski sum $C_1 + C_2$ has a $K_1 \times K_2$-lift;
\item The convex hull $\textup{conv}(C_1 \cup C_2)$ has a $K_1 \times K_2$-lift.
\end{enumerate}
\end{proposition}
\begin{proof}
The first property follows immediately from the definition of a 
$K_1$-lift. The second is an immediate consequence of
Theorem~\ref{thm:general_Yannakakis}. 
For the third property, if a face $F$ of $C_1$ is exposed, then $F = C_1 \cap H$ where $H$ is a hyperplane in $\RR^n$. If $K_1 \cap L$ is a $K_1$-lift of $C$, then $K_1 \cap L'$ is a $K_1$-lift of $F$  where $L'$ is the affine space obtained by adding the equation of $H$ to the equations defining $L$.
The fourth property is again easy
to derive from the definition since, if $C_1=\pi_1(K_1 \cap L_1)$ and
$C_2=\pi_2(K_2 \cap L_2)$, then $C_1 \times C_2 = (\pi_1 \times \pi_2)
(K_1 \times K_2 \cap L_1 \times L_2)$.  
The fifth one follows from (1) and 
the fact that the Minkowski sum $C_1+C_2$ is a linear image of the cartesian product $C_1 \times C_2$.

For the sixth, we use the fact that 
$\textup{conv}(C_1 \cup C_2)^\circ=C_1^\circ \cap C_2^\circ$. Given 
factorizations $A_1,B_1$ of $S_{C_1}$ and $A_2,B_2$ of $S_{C_2}$, we
have seen that we can extend $A_i$ to all of $C_i$, and $B_i$ to all of $C_i^{\circ}$, and get that $1-\left<x,y\right>=\left<A_i(x),B_i(y)\right>$  
for all $(x,y) \in C_i \times C_i^{\circ}$. Furthermore, extend $A_1$ to $\textup{conv}(C_1 \cup C_2)$ by
defining it to be zero outside $C_1$ and set $A_2$ to be zero outside $C_2 \setminus C_1$. 
Then, since $\ext(\textup{conv}(C_1 \cup C_2)) \subseteq \ext(C_1) \cup \ext(C_2)$ and $\ext(C_1^\circ \cap C_2^\circ)$ is contained in both $C_1^\circ$ and $C_2^\circ$, the maps, 
$(A_1,A_2):\ext(\textup{conv}(C_1 \cup C_2))\rightarrow K_1 \times K_2$ and 
$(B_1,B_2):\ext(\textup{conv}(C_1 \cup C_2)^\circ)\rightarrow K_1^*
\times K_2^*$ give a $K_1 \times K_2$ factorization of $S_{\textup{conv}(C_1 \cup C_2)}$.
\end{proof}

Explicit constructions of the lifts guaranteed in
Proposition~\ref{prop:geometric operations} can be found in the work
of Ben-Tal, Nesterov and Nemirovski; see
e.g. \cite{BenTalNemirovskiBook,NN}.  They were especially interested in
the case of lifts into the cones of positive semidefinite matrices.
Of significant interest is the relationship between lifts and duality,
particularly when considering a self-dual cone $K$.  When $K$ is
  self dual, Theorem~\ref{thm:general_Yannakakis} shows that the
  existence of a $K$-lift is a property of both the convex body and
  its polar making the theory invariant under duality. We now
  examine the behavior of cone lifts under projective transformations.

\begin{proposition} \label{prop:transformations}
Let $C \subset \RR^n$ be a convex body with a $K$-lift where $K \subset \RR^m$ is a closed convex cone. If  $\Pi$ is a projective transformation with $\Pi(C)$ compact, then $\Pi(C)$ has a $K$-lift.
\end{proposition}

\begin{proof}
Without loss of generality we may assume the lift to be proper by
passing to the smallest face of $K$ containing the lift of $C$. Then,
by Theorem \ref{thm:general_Yannakakis}, there exists maps $A \, : \,
\textup{ext}(C) \rightarrow K$ and $B \, : \, \textup{ext}(C^\circ)
\rightarrow K^*$ factorizing $S_C$, and we can extend their domains to
$C$ and $C^{\circ}$ as noted in Remark \ref{remark:extension}.  Recall
that a real projective transformation $\Pi$ in $\RR^n$ is a map
sending $x$ to $Px/(1+\left<c,x\right>)$ where $P$ is some $n \times
n$ (invertible) real matrix, and $c$ a vector in $\RR^n$ .  The
compactness of $\Pi(C)$ is equivalent to $1+\left<c,x\right>$ not
vanishing on $C$ and so we may assume without loss of generality that
$1 + \left< c, x \right>$ is positive on $C$.

Since for $y \in \Pi(C)^{\circ}$ and $x \in C$, $0 \leq 1-\left<y,
\Pi(x)\right> = 1 - \frac{y^T Px}{1+\left<c,x\right>} =
\frac{1+\left<c,x\right>-y^T Px}{1+\left<c,x\right>}$, we have that
$\left<P^T y-c,x\right> \leq 1$, and therefore, $z_y := P^T y-c \in
C^{\circ}$. Consider the maps $A': \Pi(C)\rightarrow K$ and
$B':\Pi(C)^{\circ}\rightarrow K^*$ given by $A'(x) =
A(\Pi^{-1}(x))/(1+\left<c, \Pi^{-1}(x)\right>)$ and
$B'(y)=B(z_y)$. These maps form a $K$-factorization of $S_{\Pi(C)}$
and hence, $\Pi(C)$ has a $K$-lift by Theorem
\ref{thm:general_Yannakakis}. The case of affine transformations is
trivial, but can be seen as a particular case of the projective case
we just proved.
\end{proof}

A restricted class of lifts that has received much attention is
that of {\em symmetric lifts}. The idea there is to demand that the lift not
only exists, but also preserves the symmetries of the object being
lifted. Several definitions of symmetry have been studied in the context of lifts to nonnegative orthants in papers such as \cite{Yannakakis},
\cite{KaibelPashkovichTheis} and \cite{Pashkovich}. Theorem~\ref{thm:general_Yannakakis} can
 be extended to symmetric lifts. 

Let $G$ be a subgroup of $\GL_n$ acting on $\ext(C)$. A simple example of such a group would be $\Aut(C)$, the group of all rigid linear transformations $\varphi$ of
$\RR^n$ such that $\varphi(C)=C$, restricted to $\ext(C)$. Any such group $G$ is compact, hence has a unique measure $\mu_G$, its Haar  measure, such that $\mu_{G}(G)=1$ and $\mu_{G}$ is invariant under  multiplication, i.e., $\mu_{G}(gU)=\mu_{G}(U)$ for all $g \in G$ and  all $U \subseteq G$. Note that allowing affine transformations instead of linear ones, would not be essentially different, as any group of affine transformations acting on 
a compact set has a common fixed point, so after a translation of $C$ it would be simply a subgroup of $\GL_n$. 

\begin{definition} \label{def:symmetric lift}
Let $K$ be a closed convex cone and $C$ a convex body, such that 
$C=\pi(K \cap L)$ for some affine subspace $L$ and linear map $\pi$. 
Furthermore, let $G \subseteq \GL_n$ be a group acting on $\ext(C)$ 
and $H\subseteq \GL_m$ a group acting on $K$.
We say that the lift $K \cap L$ of $C$ 
is {\em $(G,H)$-symmetric} if there exists a group homomorphism from $G$
to $H$ sending $\varphi \in G$ to $f_{\varphi} \in
H$ such that $f_{\varphi}(K \cap L)=K \cap L$ and $\pi \circ f_{\varphi} =
\varphi \circ \pi$, when restricted to $K \cap L \cap \pi^{-1}(\ext(C))$.
We will say the lift is {\em symmetric} if it is $(\Aut(C),\Aut(K))$-symmetric.
\end{definition}

The lifts obtained from the traditional lift-and-project methods mentioned in the Introduction are often symmetric in the sense of Definition~\ref{def:symmetric lift}, so it makes sense to study such lifts. 
In order to get a symmetric version of Theorem~\ref{thm:general_Yannakakis}, we have to
introduce a notion of {\em symmetric factorization} of $S_C$.

\begin{definition} \label{def:symmetric factorization}
Let $C$, $K$, $G$ and $H$ be as in Definition~\ref{def:symmetric lift}, and
$A:\ext(C)\rightarrow K$ and $B:\ext(C^\circ)\rightarrow K^*$ a
$K$-factorization of $S_C$. We say that the factorization is {\em  $(G,H)$-symmetric}
if there exists a group homomorphism from $G$ to $H$
sending $\varphi \in G$ to $f_{\varphi} \in H$ such that
$A \circ \varphi = f_{\varphi} \circ A$.
Call the factorization {\em symmetric} if it is $(\Aut(C),\Aut(K))$-symmetric.
\end{definition}

Note that any action of $G \subseteq \GL_n$ on $C$ defines trivially an action of $G$ on $C^{\circ}$, and similarly any action of $H\subseteq \GL_m$ on $K$ defines an
action on $K^*$. With these actions we can see that if a $K$-factorization is $(G,H)$-symmetric in the sense of the previous definition, the group homomorphism $f$ would also verify
$B \circ \varphi = f_{\varphi} \circ B$. Hence, Definition~\ref{def:symmetric factorization} is actually invariant with respect to polarity, even if it seems to only depend on the map $A$.
This observation would still be true if we had considered $G$ and $H$ to be subgroups of projective transformations of $\RR^n$ and $\RR^m$ respectively, but general linear
groups are enough to cover all interesting examples we know.
We can now establish the symmetric version of Theorem~\ref{thm:general_Yannakakis}.

\begin{theorem} \label{thm:general_Yannakakis_symm}
If $C$ has a proper $(G,H)$-symmetric $K$-lift then $S_C$ has a $(G,H)$-symmetric $K$-factorization. 
Conversely, if $S_C$ has a $(G,H)$-symmetric $K$-factorization then $C$ has a $(G,H)$-symmetric $K$-lift.
\end{theorem}
\begin{proof} 
First suppose that $C$ has a proper $(G,H)$-symmetric $K$-lift with $C=\pi(K \cap L)$. 
For each orbit of the action of the group $G$ on $\ext(C)$, pick a representative $x_0$, and
let $A'(x_0)$ be any point in $K \cap L$ such that $\pi(A'(x_0))=x_0$. Let $G_{x_0} \subseteq G$ be the subgroup of all automorphisms
that fix $x_0$. Then we can define 
$$A(x_0) := \int_{\varphi \in G_{x_0}} f_{\varphi}(A'(x_0)) d\mu_{G_{x_0}}$$
which generalizes the construction in \cite[Step 2, pp 449]{Yannakakis}.
For a finite group, this is just the usual average of all images of $A'(x_0)$ under the action of $G_{x_0}$. 
For any other point $x'$ in the same orbit as $x_0$, pick any $\psi$ such that $\psi(x_0)=x'$ and define $A(x') := f_{\psi}(A(x_0))$. The point $A(x')$ in $K \cap L$  does
not actually depend on the choice of $\psi$. To see this it is enough to note that $f_{\mu}\circ A(x_0) = A(x_0)$ for all $\mu \in G_{x_0}$
and if $\psi_1$ and $\psi_2$ both send $x_0$ to $x'$, then $f_{\psi_1}^{-1} \circ f_{\psi_2}=f_{\psi_1^{-1}\psi_2}$ and $\psi_1^{-1}\psi_2$ is in
$G_{x_0}$.

Since $K \cap L$ is a proper lift of $C$, we know we have a $K$-factorization of $S_C$ by Theorem~\ref{thm:general_Yannakakis}. If we follow the proof of that 
result, we see that it is actually constructive, in the sense that we can pick as a map from $\ext(C) \rightarrow K$ any section of the projection $\pi$. In particular,
we can pick the map $A$ we just defined, since we have $\pi(A(x))=x$ for every $x \in \ext(C)$. This means that such a map $A$ can be extended to a $K$-factorization $A,B$ of $S_C$.
 For any $\mu \in G$ and $x \in \ext(C)$, we have
$A \circ \mu(x) = A\circ \mu \circ \psi (x_0)$, for some $\psi$ and $x_0$ in the orbit of $x$ and so, by the above considerations,
$$A \circ \mu(x) = f_{\mu \circ \psi} \circ A(x_0) = f_{\mu} \circ f_{\psi} \circ A(x_0)=f_{\mu}\circ A(\psi x_0)=f_{\mu} \circ A(x),$$
and hence, we have a $(G,H)$-symmetric $K$-factorization of $S_C$.

Suppose now we have a $(G,H)$-symmetric $K$-factorization of $S_C$. Since it is in particular a $K$-factorization of $S_C$, we have a $K$-lift $K \cap L$ of $C$ by
Theorem~\ref{thm:general_Yannakakis}. From the proof of that theorem we know that $A(x)$ is in 
 $K \cap L$ for all $x \in \ext(C)$.
Let $L'$ be the affine subspace of $L$ spanned by all such points $A(x)$. It is clear from the definition that $L'$ is $f_{\varphi}$ invariant for
all $\varphi \in G$. Furthermore, given any $y \in L'$ we can write it as an affine combination $\sum_i \alpha_i A(x_i)$ for some
$x_i$ in $\ext(C)$, and so for all $\varphi \in G$, we have
$$\pi(f_{\varphi}(y)) = \sum_{i} \alpha_i \pi(f_{\varphi}(A(x_i))) = \sum_{i} \alpha_i \pi(A(\varphi x_i)) = \sum_{i} \alpha_i \varphi x_i,$$
which is simply the image of $\pi(y)$ under $\varphi$. Hence, $K \cap L'$ is a $(G,H)$-symmetric lift of $C$.
\end{proof}

\section{Cone lifts of polytopes} 
\label{sec:polytopes}

The results developed in the previous section for general convex bodies specialize nicely to polytopes, 
providing a more general version of the original result of Yannakakis relating polyhedral lifts of polytopes 
and nonnegative factorizations of their slack matrices. We first introduce the necessary definitions.

For a full-dimensional polytope $P$ in $\RR^n$, let
$V_P=\{p_1,\ldots,p_v\}$ be its set of vertices, $F_P$ its set of
facets, and $f := \left|F_P\right|$. Recall that each facet $F_i$ in
$F_P$ corresponds to a unique (up to multiplication by nonnegative
scalars) linear inequality $h_i(x) \geq 0$ that is valid on $P$ such that 
$F_i=\{ x \in P : h_i(x)=0\}$. These form 
(again up to multiplication by nonnegative scalars) the unique 
irredundant representation of $P$ as
\[
P=\{ x \in \RR^n : h_1(x) \geq 0 , \ldots, h_f(x) \geq 0\}.
\] 
Since we are assuming that the origin is in
the interior of $P$, $h_i(0) > 0$ for each $i=1,\ldots,f$. Therefore, we can make the facet description
of $P$ unique by normalizing each $h_i$ to verify $h_i(0)=1$. We will call
this the {\em canonical inequality representation of $P$}.

\begin{definition}
Let $P$ be a full-dimensional polytope in $\RR^n$ with vertex set 
$V_P=\{p_1,\ldots,p_v\}$ and with an inequality representation
\[
P=\{ x \in \RR^n : h_1(x) \geq 0 , \ldots, h_f(x) \geq 0\}.
\]
Then the nonnegative matrix in $\RR^{v \times f}$ whose $(i,j)$-entry
is $h_j(p_i)$ is called a {\em slack matrix of $P$}. If the $h_i$ form
the canonical inequality representation of $P$, we call the corresponding slack matrix 
the {\em canonical slack matrix of $P$}.
\end{definition}

In the case of a polytope $P$, $\ext(P)$ is just $V_P$, and the
elements of $\ext(P^\circ)$ are in bijection with the facets of
$P$. This means that the operator $S_P$ is actually a finite map from
$V_P \times F_P$ to $\RR_+$ that sends a pair $(p_i,F_j)$ to
$h_j(p_i)$, where $h_j$ is the canonical inequality corresponding to
the facet $F_j$.  Hence,we may identify the slack operator of $P$ with
the canonical slack matrix of $P$ and use $S_P$ to also denote this
matrix.  We now need a definition about factorizations of non-negative
matrices.

\begin{definition} 
\label{def:K-factorization of nonnegative matrices}
Let $M = (M_{ij}) \in \RR_+^{p \times q}$ be a nonnegative matrix and $K$ a
closed convex cone. Then a $K$-{\em factorization} of $M$ is a pair of
ordered sets $a^1, \ldots, a^p \in K$ and $b^1, \ldots, b^q \in K^*$
such that $\langle a^i, b^j \rangle = M_{ij}$.
\end{definition}

Note that $M \in \RR^{p \times q}_+$ has a $\RR_+^k$-factorization if and only
if there exist a $p \times k$ nonnegative matrix $A$ and a $k \times
q$ nonnegative matrix $B$ such that $M=AB$. Therefore, Definition~\ref{def:K-factorization of nonnegative matrices} 
generalizes nonnegative factorizations of nonnegative matrices to arbitrary closed convex cones.  Since any slack matrix of $P$ can be obtained from the
canonical one by multiplication by a diagonal nonnegative matrix,
it is $K$-factorizable if and only if $S_P$ is $K$-factorizable. 
We can now state Theorem~\ref{thm:general_Yannakakis} for polytopes.

\begin{theorem} \label{thm:general_Yannakakis_for_polytopes}
If a full-dimensional polytope $P$ has a proper $K$-lift then every slack matrix of $P$ admits a 
$K$-factorization. Conversely, if some slack matrix of $P$ has a
$K$-factorization then $P$ has a $K$-lift.
\end{theorem}

Theorem~\ref{thm:general_Yannakakis_for_polytopes} is a direct translation of 
Theorem~\ref{thm:general_Yannakakis} 
using the identification between the slack operator of $P$ and the canonical slack
matrix of $P$. The original theorem of Yannakakis \cite[Theorem~3]{Yannakakis} proved this result 
in the case where $K$ was some nonnegative orthant $\RR_+^l$.

\begin{example} \label{ex:hexagon}
To illustrate Theorem~\ref{thm:general_Yannakakis_for_polytopes} consider the regular hexagon in the plane with canonical inequality description 
$$ H=\left \{ (x_1,x_2) \in \RR^2 \,:\, \left(\begin{array}{cc}
1 & \sqrt{3}/3 \\
0 & 2\sqrt{3}/3 \\
-1 & \sqrt{3}/3 \\
-1 & -\sqrt{3}/3 \\
0 & -2\sqrt{3}/3 \\
1 & -\sqrt{3}/3 
\end{array}
\right)
\left(\begin{array}{c}
x_1 \\
x_2
\end{array}
\right) \leq 
\left(
\begin{array}{c}
1 \\
1\\
1 \\  
 1 \\
 1\\
 1
\end{array}
\right) \right \}. $$ 
We will denote the coefficient matrix by $F$ and the right hand side
vector by $d$.  It is easy to check that $H$ cannot be the projection of an 
affine slice of $\RR^k_+$ for $k < 5$. Therefore, we ask whether it can be the linear image 
of an affine slice of $\RR_+^5$, which turns out to be 
surprisingly non-trivial. Using
Theorem~\ref{thm:general_Yannakakis_for_polytopes} this is equivalent to asking if 
the canonical slack matrix of the hexagon, 
$$ S_H := \left( \begin{array}{cccccc} 
0 & 0 & 1 & 2 & 2 & 1 \\
1 & 0 & 0 & 1 & 2 & 2\\
2 & 1 & 0 & 0 & 1 & 2 \\
2 & 2 & 1 & 0 & 0 & 1 \\
1 & 2 & 2 & 1 & 0 & 0 \\
0 & 1 & 2 & 2 & 1 & 0
\end{array} \right),$$
has a $\RR_+^5$-factorization. Check that 
$$S_H = 
\left( \begin {array}{ccccc} 1&0&1&0&0\\\noalign{\medskip}1&0&0&0&1
\\\noalign{\medskip}0&0&0&1&2\\\noalign{\medskip}0&1&0&0&1\\\noalign{\medskip}0
&1&1&0&0\\\noalign{\medskip}0&0&2&1&0\end {array} \right) 
\left( \begin {array}{cccccc} 0&0&0&1&2&1\\\noalign{\medskip}1&2&1&0&0&0
\\\noalign{\medskip}0&0&1&1&0&0\\\noalign{\medskip}0&1&0&0&1&0
\\\noalign{\medskip}1&0&0&0&0&1\end {array} \right),$$
where we call the first matrix $A$ and the second matrix $B$. We may take the rows of $A$ 
as elements of $\RR^5_+$, and the columns of $B$ as elements of $\RR_+^5 = (\RR^5_+)^*$, and they provide us a $\RR_+^5$-factorization
of the slack matrix $S_H$, proving that this hexagon has a $\RR^5_+$-lift while the trivial polyhedral lift would have been to $\RR^6_+$.

We can construct the lift explicitly using the proof of the Theorem~\ref{thm:general_Yannakakis}. Note that 
$$H=\{(x_1,x_2) \in \RR^2 : \exists \,\,y \in \RR_+^5 \textrm{ s.t. } F x + B^T y = d\}.$$
Hence, the exact slice of $\RR_+^5$ that is mapped to the hexagon is simply
$$\{y \in \RR_+^5 : \exists \,\,x \in \RR^2 \textrm{ s.t. }  B^T y = d - F x\}.$$
By eliminating the $x$ variables in the system we get 
$$\{ y \in \RR^5_+ \,:\, y_1 + y_2 + y_3 + y_5 = 2, y_3 + y_4 + y_5 = 1 \},$$
and so we have a three dimensional slice of $\RR^5_+$ projecting down to $H$. This projection is visualized
in Figure \ref{fig:hexagon_proj}.

\begin{figure} 
\includegraphics[scale=0.5]{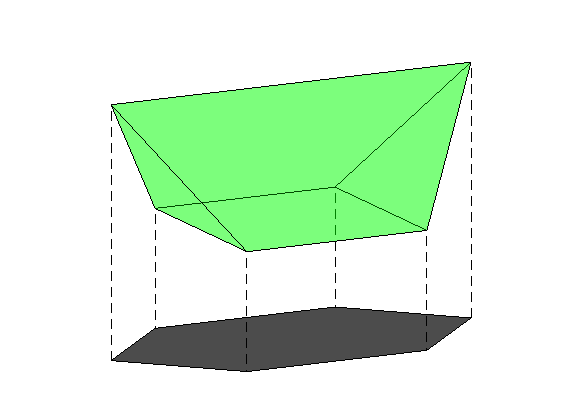}
\caption{Lift of the regular hexagon.}
\label{fig:hexagon_proj}
\end{figure}

The hexagon is a good example to see that the existence of lifts depends on more than the combinatorics of the facial structure of the polytope.
If instead of a regular hexagon we take the hexagon with vertices $(0,-1)$, $(1,-1)$, $(2,0)$, $(1,3)$, $(0,2)$ and $(-1,0)$, as seen in Figure \ref{fig:irregular_hexagon}, a valid slack matrix would be
$$ S := \left(
\begin{array}{cccccc}
 0 & 0 & 1 & 4 & 3 & 1 \\
 1 & 0 & 0 & 4 & 4 & 3 \\
 7 & 4 & 0 & 0 & 4 & 9 \\
 3 & 4 & 4 & 0 & 0 & 1 \\
 3 & 5 & 6 & 1 & 0 & 0 \\
 0 & 1 & 3 & 5 & 3 & 0
\end{array}
\right).$$
One can check that if a $6 \times 6$ matrix with the zero pattern of a slack matrix of a hexagon has a $\RR^5_+$-factorization, then it has a factorization with either the same zero pattern as the matrices $A$ and $B$ obtained
before, or the patterns given by applying a cyclic permutation to the rows of $A$ and the columns of $B$.
A simple algebraic computation then shows that the slack matrix $S$ above has no such decomposition hence this irregular hexagon has no $\RR_+^5$-lift.

\begin{figure} 
\includegraphics[scale=0.4]{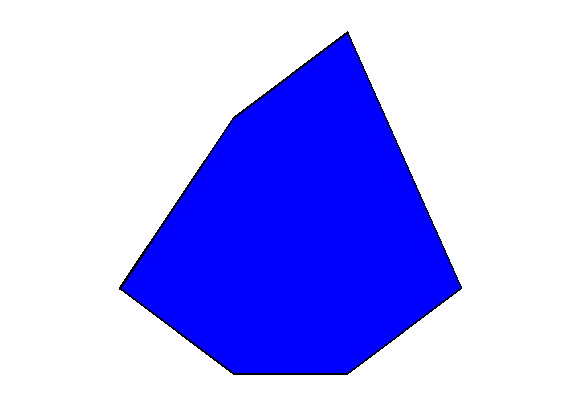}
\caption{Irregular hexagon with no $\RR^5_+$-lift.}
\label{fig:irregular_hexagon}
\end{figure}
\end{example}


\medskip

Symmetric lifts of polytopes are especially interesting to study since the automorphism group of a polytope is finite. We now show that there are polygons with $n$ sides for which a symmetric $\RR^k_+$-lift requires $k$ to be at least $n$. 

\begin{proposition} \label{prop:prime gons}
A regular polygon with $n$ sides where $n$ is either a prime number or a power of a prime number cannot admit a symmetric 
$\RR^k_+$-lift where $k < n$.
\end{proposition}

\begin{proof}
A  symmetric $\RR^k_+$-lift of a polytope $P$ implies the existence of an injective group homomorphism from $\Aut(P)$ to $\Aut(\RR^k_+)$. Since the rigid transformations of $\RR^k_+$ are the permutations of coordinates, $\Aut(\RR_+^k)$ is the symmetric group $S_k$. This implies that the cardinality of $\Aut(P)$ must divide $k!$. 

Let $P$ be a regular $p$-gon where $p$ is prime. Since $\Aut(P)$ has $2p$ elements, and the smallest $k$ such that $2p$ divides $k!$ is $p$ (since $p > 2$), we can never do better than a symmetric $\RR_+^p$-lift for $P$. 
If $P$ is a $p^t$-gon, then the homomorphism from $\Aut(P)$ to $S_k$ must send an element of order $p^t$ to an element whose order is a multiple of $p^t$. The smallest symmetric group with an element of order $p^t$ is $S_{p^t}$ and hence, $P$ cannot have a symmetric $\RR^k_+$-lift with $k < p^t$.
\end{proof}

In Example~\ref{ex:hexagon} we saw a $\RR^5_+$-lift of a regular hexagon, but notice that the accompanying factorization is not symmetric.

\begin{remark}
Ben-Tal and Nemirovski have shown in \cite{BenTalNemirovski} that a regular $n$-gon admits a $\RR^k_+$-lift where $ k = O(\log \,n)$. Combining their result with Proposition~\ref{prop:prime gons} provides a simple family of polytopes where there is an exponential gap between the sizes of the smallest possible symmetric and non-symmetric lift into nonnegative orthants. This provides a simple illustration of the impact of symmetry on the size of lifts, a phenomenon that was investigated in detail by Kaibel, Pashkovich and Theis in \cite{KaibelPashkovichTheis}.
\end{remark}

\section{Cone ranks of convex bodies} \label{sec:coneranks}

In Section~\ref{sec:mainthms} we established necessary and sufficient conditions for the existence of a $K$-lift of a given convex body $C \subset \RR^n$ for a fixed cone $K$. In many instances, the cone $K$ belongs to a family such as $( \RR^i_+ )_i$ or $( \PSD^i )_i$. In such cases, it becomes interesting to determine the smallest cone in the family that admits a lift of $C$. In this section, we study this scenario and develop the notion of {\em cone rank} of a convex body.

\subsection{Definitions and basics}
\label{subsec:cone ranks defs}

\begin{definition}
A {\em cone family} $\K = (K_i)_{i \in \NN}$ is a sequence of closed convex cones $K_i$ indexed by $i \in \NN$. The family $\K$ is said to be {\em closed} if for every $i \in \NN$
and every face $F$ of $K_i$ there exists $j \leq i$ such that $F$ is isomorphic to $K_j$.
\end{definition}

\begin{example}
\begin{enumerate}
\item[]
\item The set of nonnegative orthants $( \RR_+^i, \ i \in \NN  )$ form a closed cone family.
\item The family $( \PSD^i, i \in \NN )$ where $\PSD^i$ is the set of all $i \times i$ real symmetric positive semidefinite matrices is  closed since every face of $\PSD^i$ is isomorphic
to a $\PSD^j$ for $j \leq i$ \cite[Chapter II.12]{Barvinok}.
\item Recall that a $i \times i$ symmetric matrix $A$ is {\em copositive} if 
$x^TAx \geq 0$ for all $x \in \RR^i_+$. Let the cone of $i \times i$ symmetric copositive
matrices be denoted as $C_i$.  This family is not closed --- the set of all $i \times i$ matrices with zeroes on the diagonal and nonnegative off-diagonal entries form a face of $C_i$ that is isomorphic to the nonnegative orthant of dimension ${i \choose 2}$.
\item The dual of $C_i$ is the cone $C_i^*$ of all {\em completely positive} matrices which are exactly those symmetric $i \times i$ matrices that factorize as $BB^T$ for some $B \in \RR_+^{i \times k}$. 
The family $( C_i^*, \, i \in \NN )$ is also not closed since $\dim C^*_i = {i \choose 2}$ while $C_i^*$ 
has facets (faces of dimension ${i \choose 2} -1$) which therefore, cannot belong to the family.
\end{enumerate}
\end{example}

Recall the definition of a cone factorization of a nonnegative matrix in Definition~\ref{def:K-factorization of nonnegative matrices}.

\begin{definition} \label{def:cone ranks}
Let $\K=(K_i)_{i\in \NN}$ be a closed cone family.
\begin{enumerate}
\item The $\K$-rank of a nonnegative matrix $M$, 
denoted as $\rank_{\K}(M)$, is the smallest $i$ such that
$M$ has a $K_i$-factorization. If no such $i$ exists,  
we say that $\rank_{\K}(M) = +\infty$.
\item The $\K$-rank of a convex body $C \subset \RR^n$, denoted as 
$\rank_\K(C)$, is the 
smallest $i$ such that the slack operator $S_C$ has a $K_i$-factorization.
If such an $i$ does not exist, we say that $\rank_{\K}(C) = +\infty$.
\end{enumerate}
\end{definition}

In this paper, we will be particularly interested in the families $\K = ( \RR^i_+ )$ and 
$\K = ( \PSD^i )$. In the former case, we set 
$\rank_+(\cdot) := \rank_\K(\cdot)$ and call it {\em nonnegative rank}, and in the latter case we set $\rank_{\textup{psd}}(\cdot) := \rank_{\K}(\cdot)$ and call it {\em psd rank}. 
Our interest in cone ranks comes from their connection to the existence of cone lifts. The following is immediate from Theorem~\ref{thm:general_Yannakakis}.

\begin{theorem}
Let $\K = (K_i)_{i \geq 0}$ be a closed cone family and $C \subset \RR^n$ a convex body. 
Then $\rank_{\K}(C)$ is the smallest $i$ such that $C$ has a $K_i$-lift.
\end{theorem}
\begin{proof}
If $i = \rank_{\K}(C)$, then we have a $K_i$-factorization of the slack operator $S_C$, and therefore, by Theorem~\ref{thm:general_Yannakakis}, $C$ has a $K_i$-lift. Take the smallest $j$ for which 
$C$ has a $K_j$-lift and suppose $j<i$. If the lift was proper, we would get a $K_j$ factorization of $S_C$ for $j < i$, which contradicts that $i = \rank_{\K}(C)$. Therefore, the $K_j$-lift of $C$ is not proper, and $C$ has a lift to a proper face of $K_j$. Since $\K$ is closed, this would imply a $K_{l}$-lift of $C$ for $l<j$ contradicting the definition of $j$.
\end{proof}

In practice one might want to consider lifts to products of cones in a family. This could be dealt with by defining rank as the tuple of indices of the factors in such a product, minimal under some order. In this paper we are mostly working with the families $(\RR^i_+)$ and $(\PSD^i)$, and in the first case,  $\RR_+^n \times \RR_+^m = \RR_+^{n+m}$, and in the second case, $\PSD^n \times \PSD^m = \PSD^{m+n} \cap L$
where $L$ is a linear space. Therefore, in these situations, there is no incentive to consider lifts to products of cones.
However, if one wants to study lifts to the family of second order cones, considering products of cones makes sense. 

Having defined $\textup{rank}_{\K}(M)$ for a nonnegative matrix $M$, it is natural to ask how it compares with the usual rank of $M$. We now look at this relationship for the nonnegative and psd ranks of a nonnegative matrix. 

The nonnegative rank of a nonnegative matrix arises in several
contexts and has wide applications \cite{CohenRothblum}. As mentioned earlier, its relation to
$\RR^k_+$-lifts of a polytope was studied by Yannakakis
\cite{Yannakakis}.  Determining the nonnegative rank of a matrix is
NP-hard in general \cite{Vavasis}, but there are obvious upper and lower
bounds on it.

\begin{lemma}
\label{prop:trivial_bounds}
For any $M \in \RR_+^{p \times q}$, $\rank(M) \leq \rank_+(M) \leq
\min\{p,q\}$.
\end{lemma}

Further, it is not possible in general, to bound
$\rank_+(M)$ by a function of $\rank(M)$.

\begin{example} 
\label{ex:rank_vs_nonneg_gap}
Consider the $n \times n$ matrix $M_n$ whose $(i,j)$-entry is
$(i-j)^2$. Then $\rank(M_n)=3$ for all $n$ since $M_n = A_n B_n$ where
row $i$ of $A_n$ is $(i^2 , -2i , 1)$ for $i=1, \ldots, n$ and column
$j$ of $B_n$ is $(1 , j , j^2)^T$ for $j=1,\ldots,n$.  If $M_n$ has a
$\RR^k_+$-factorization, then there exists $a_1, \ldots, a_n, b_1,
\ldots, b_n \in \RR^k_+$ such that $\langle a_i, b_j \rangle \neq 0$
for all $i \neq j$. Notice that for $i \neq j$ if $\supp(b_j)
\subseteq \supp(b_i)$ then $\langle a_i , b_i \rangle =0$ implies
$\langle a_i, b_j \rangle = 0$, and hence, all the $b_i$'s (and also
all the $a_i$'s) must have supports that are pairwise incomparable. By
Sperner's lemma, the largest antichain in the Boolean lattice of
subsets of $[k]$ has cardinality ${k \choose \lfloor \frac{k}{2}
  \rfloor}$, and thus we get that $n \leq {k \choose \lfloor
  \frac{k}{2} \rfloor}$. Therefore, $\rank_+(M_n)$ is bounded below by
the smallest integer $k$ such that $n \leq {k \choose \lfloor
  \frac{k}{2} \rfloor}$. For large $k$, we have ${k \choose \lfloor
  \frac{k}{2} \rfloor} \approx \sqrt{\frac{2}{\pi k}} \cdot 2^k$, and
the easy bound ${k \choose \lfloor \frac{k}{2} \rfloor} \leq 2^k$
yields $\rank_+(M_n) \geq \log_2 n$.
\end{example}

The psd rank of a nonnegative matrix is connected to rank and $\rank_+$ as follows.

\begin{proposition} \label{prop:upper and lower bounds for psd rank}
For any nonnegative matrix $M$
$$ \frac{1}{2}\sqrt{1+8 \,\rank{(M)}}-\frac{1}{2} \leq \rank_{psd}(M) \leq \rank_+(M).$$
\end{proposition}

\begin{proof}
Suppose $a_1,\ldots,a_p,b_1,\ldots,b_q$ give a $\RR_+^r$-factorization of $M \in \RR^{p \times q}_+$. Then the diagonal matrices $A_i:= \diag(a_i)$ and
$B_j :=\diag(b_j)$ give a $\PSD^r$-factorization of $M$, and we obtain the second inequality.

Now suppose $A_1, \ldots A_p, B_1, \ldots, B_q $ give a $\PSD^r$-factorization of $M$.
Consider the vectors
$$a_i=(A_{11},\ldots,A_{rr},2A_{12},\ldots,2A_{1r},2A_{23}, \ldots, 2A_{(r-1)r})$$
and 
$$b_j=(B_{11},\ldots,B_{rr},B_{12},\ldots,B_{1r},B_{23}, \ldots, B_{(r-1)r})$$
in $\RR^{r+1 \choose 2}$ where $A = A_i$ and $B = B_j$. Then $\langle a_i,b_j \rangle =  \langle A_i, B_j \rangle= M_{ij}$ so $M$ has rank at most ${r+1 \choose 2}$. By solving for $r$
we get the desired inequality.
\end{proof}

There is a simple, yet important situation where 
$\rank(M)$ is an upper bound on $\rank_{\textup{psd}}(M)$.

\begin{proposition} 
\label{prop:squaring}
Take $M \in \RR^{p \times q}$ and let $M'$ be the nonnegative
matrix obtained from $M$ by squaring each entry of $M$. Then
$\rank_{psd}(M') \leq \rank(M)$. In particular, if $M$ is a $0/1$
matrix, $\rank_{psd}(M) \leq \rank(M)$.
\end{proposition}

\begin{proof}
Let $\rank(M)=r$ and $v_1,\ldots
,v_p, w_1,\ldots, w_q \in \RR^r$ be such that $\langle v_i, w_j
\rangle =M_{ij}$. Consider the matrices $A_i = v_i v_i^T$, $i=1,\ldots,p$ and $B_j = w_j w_j ^T$, $j=1,\ldots,q$ in $\PSD^r$. Then, since $\langle A_i, B_j \rangle = \langle v_i,w_j \rangle^2 =
M'_{ij}$, the matrix $M'$ has a $\PSD^r$-factorization.
\end{proof}

Barvinok has generalized the above result in a recent preprint \cite{Barvinok2012} to show that 
when the number of distinct entries in a nonnegative matrix $M$ does not exceed $k$, then the psd rank of $M$ is bounded above by 
${ {k-1 + \textup{rank}(M) } \choose {k-1}}$.
We now see that the gap between the nonnegative and psd rank
of a nonnegative matrix can become arbitrarily large. 

\begin{example} 
\label{ex:psd_rank_vs_nonneg_gap}
Let $E_n$ be the $n \times n$ matrix, $n \geq 2$, whose $(i,j)$-entry is
$i-j$. Then $\rank(E_n) = 2$ since the vectors $a_i := (i,-1)$, $i=1,\ldots,n$ and $b_j = (1,j)$, $j=1,\ldots,n$ have the property that $\langle a_i, b_j \rangle = i-j$. Therefore, by 
Proposition~\ref{prop:squaring}, the matrix $M_n$ with $(i,j)$-entry
equal to $(i-j)^2$ has psd rank two and an explicit $\PSD^2$-factorization of $M_n$ is given by the psd matrices 
$$A_i := \left( \begin{array}{rr} i^2 & -i \\ -i & 1 \end{array} \right), \,\,\,i=1,\ldots, n \mbox{ and } B_j := \left( \begin{array}{rr} 
1 & j \\ j & j^2 \end{array} \right), \,\,\,j=1,\ldots,n.$$
However, we saw in
Example~\ref{ex:rank_vs_nonneg_gap} that $\rank_+(M_n)$ grows with
$n$. A family of $n \times n$ matrices for which psd rank is $\textup{O}(\textup{log} \,n)$ and nonnegative rank at least $n^{\textup{constant}}$ is given in \cite{FMPTW}. For the family 
$\{ M_n \}$, the gap between rank and psd rank can become arbitrary large.
\end{example}

Thus, so far we have seen that the gap between $\rank(M)$ and $\rank_+(M)$ as well as the gap between $\rank_{\textup{psd}}(M)$ and $\rank_+(M)$ can be made arbitrarily large for nonnegative matrices $M$. Results in the next subsection will imply that there are nonnegative matrices for which the gap between $\rank(M)$ and $\rank_{\textup{psd}}(M)$ can also become arbitrarily large.


\subsection{Lower bounds on the nonnegative and psd ranks of polytopes}
\label{subsec:cone ranks of convex sets}

A well-known lower bound to the nonnegative rank of a matrix is the Boolean rank of 
the support of the matrix. The {\em support} of a matrix $M \in \RR^{p \times q}_+$, is the Boolean matrix 
$\supp(M)$ obtained by turning every non-zero entry in $M$ to a one. The rank of $\supp(M)$ in 
Boolean arithmetic (where $1+1 = 1$ and all other additions and multiplications among $0$ and $1$ are as for the integers) 
is called the {\em Boolean rank} of $\supp(M)$ (and also of $M$). In terms of factorizations, Boolean rank can be defined as follows.

\begin{definition} \label{def:Boolean rank}
The {\em Boolean rank} of a matrix $T \in \{0,1\}^{p \times q}$ is the least integer $r$ for which there exists
$A \in \{0,1\}^{p \times r}$ and $B \in \{0,1\}^{r \times q}$ such that $T = AB$ where all additions and multiplications are 
in Boolean arithmetic. 
\end{definition}

We will denote the Boolean rank of $\supp(M)$ as $\rank_B(M)$. It is easy to see that $\rank_B(M) \leq \rank_+(M)$. However, it is NP-hard to compute Boolean rank and 
most lower bounds to $\rank_+(M)$ are, in fact, lower bounds to $\rank_B(M)$. 

The ideas in Example~\ref{ex:rank_vs_nonneg_gap} provide an elegant way of thinking about 
lower bounds for the nonnegative rank of a polytope. Let $C$ be a polytope
and let $L(C)$ be its face lattice. If $C$ has a lift as $C=\pi(\RR_+^k \cap L)$, then the map $\pi^{-1}$ sends
faces of $C$ to faces of $\RR_+^k \cap L$. Since each face of $\RR_+^k \cap L$ is the intersection of a 
face of $\RR_+^k$ with $L$, the map $\pi^{-1}$ is an injection from $L(C)$ to the faces of $\RR_+^k$. 
The faces of $\RR^k_+$ can be identified with subsets of $[k]$ as they are of the form 
$F_J=\{ x \in \RR_+^k : \supp({x}) \subseteq J\}$ for
$J \subseteq [k]$. So the map $\pi^{-1}$ determines an embedding of the 
lattice $L(C)$ into $2^{[k]}$, the Boolean lattice of subsets of $[k]$.

\begin{theorem} \label{thm:boolean rank and lattice emdeddings}
For a polytope $C$, there is a Boolean factorization of $\supp(S_C)$ of intermediate dimension $k$  if and only if there is a lattice embedding of $L(C)$ into $2^{[k]}$. 
\end{theorem}

\begin{proof}
In this proof it is convenient to identify a subset $U$ of $[k]$ with its incidence vector in $\{0,1\}^k$ defined as having $1$ in position $i$ if and only if $i \in U$.
Given an embedding $\phi$ of $L(C)$ into $2^{[k]}$, a Boolean factorization $AB$ of $\supp(S_C)$ is 
gotten by taking the row of $A$ indexed by vertex $v$ of $C$ to be $\phi(v)$, and 
the column of $B$ indexed by facet $F$ of $C$ to be $[k] \backslash \phi(F)$. 
Then the $(v,F)$ entry of $\supp(M)$ is zero if and only if $\phi(v) \subseteq \phi(F)$ if and
only if $v \in F$.

Suppose now we have a Boolean factorization $AB$ of $\supp(S_P)$ of intermediate dimension $k$. For every face $F$ of $P$ define $$\phi(F) :=\bigcup_{v \in F} A(v)$$
where $A(v)$ denotes the row of $A$ indexed by vertex $v$.  Clearly $H \subseteq F$ implies $\phi(H) \subseteq \phi(F)$. To see the reverse inclusion, suppose $H \not \subseteq F$. Pick a vertex $w \in H\setminus F$ and a facet $\tilde{F}$ containing $F$ but not $w$. Let $B(\tilde{F})$ denote the column of $B$ indexed by facet $\tilde{F}$. Since $A(w) \cap B(\tilde{F}) \not = \emptyset$, we have $\phi(H) \cap B(\tilde{F}) \not = \emptyset$. On the other hand, for all $v \in F$, we have $v \in \tilde{F}$ which implies that $A(v) \cap B(\tilde{F})= \emptyset$ and so, $\phi(F) \cap B(\tilde{F})= \emptyset$. Therefore, $\phi(H) \not \subseteq \phi(F)$, completing the proof.
\end{proof}

Theorem~\ref{thm:boolean rank and lattice emdeddings} immediately yields a 
lower bound on the nonnegative rank of a polytope based solely on the facial structure of the polytope.

\begin{corollary} 
\label{cor:lattice homomorphism}
Let $C \subset \RR^n$ be a polytope and $k$ the smallest integer such that there exists an embedding of 
the face lattice $L(C)$ into the Boolean lattice $2^{[k]}$. Then $\rank_+(C) \geq k$. 
\end{corollary}

The Boolean rank of a $0/1$ matrix is also called its {\em rectangle covering number}. 
Theorem 2.9 in \cite{FKPT} phrases a version of the above results in terms of rectangle covering number.

\begin{corollary} \label{cor:bounds on nonneg rank examples}
If $C \subset \RR^n$ is a polytope, then the following hold:
\begin{enumerate}
\item Let $p$ be the size of a largest {\em antichain} of faces of $C$ (i.e., a largest set of faces such that no one is contained in another). Then $\rank_+(C)$ is
bounded below by the smallest $k$ such that $p \leq {k \choose \lfloor \frac{k}{2} \rfloor}$;
\item (Goemans \cite{Goemans}) Let $n_C$ be the number of faces of $C$, then $\rank_+(C) \geq \log_2(n_C).$
\end{enumerate}
\end{corollary}
\begin{proof}
The first bound follows from Corollary~\ref{cor:lattice
  homomorphism} since lattice embeddings preserve
antichains, and the size of the largest antichain of the Boolean
lattice $2^{[k]}$ is ${k \choose \lfloor \frac{k}{2} \rfloor}$
(Sperner's lemma). The second bound follows from the easy fact that
any embedding of $L(C)$ into $2^{[k]}$ requires $\# L(C) \leq
2^k$.
\end{proof}

Note that a (weaker) version of the first bound can be found in
  \cite[Corollary 4]{GillisGlineur} with the size of the largest
  antichain replaced by the number of vertices.  As mentioned, the
  second lower bound essentially appears in \cite{Goemans}. Further lower
  bounds for the nonnegative rank of a polytope are overviewed in
  \cite{FKPT}. The two bounds in Corollary~\ref{cor:bounds on nonneg rank examples}
are in general different. For instance, if $C$ is a square in the
plane, the Goemans bound says that $\rank_+(C) \geq 
\textup{log}_2(10) \sim 3.32$ while the antichain bound says that
$\rank_+(C) \geq 4$, and thus both give the same value after rounding
up. For $C$ a three-dimensional cube, $\textup{log}_2(28) = 4.807355$
while the maximum size of an antichain of faces is 12 (take the 12
edges) and hence, the antichain lower bound is $6$. Although the
antichain bound can be better than Goemans' (as this example shows),
asymptotically they are roughly equivalent. To see this, we notice
that if $ p \approx \binom{k}{\lfloor \frac{k}{2} \rfloor}$, then an
asymptotic expansion yields $k \approx C_1 + \log_2 p + \frac{1}{2}
\log_2 (C_2 + 2 \log p)$, for some small explicit constants $C_1$ and
$C_2$. Since $p$ (antichain size) is always less than or equal to
the number of faces $n_C$, we have $\log_2 p \leq \log_2 n_C$, and
thus the antichain bound is at most an additive logarithmic term
greater than the Goemans bound.

We close the study of nonnegative ranks with a family of polytopes for which all slack matrices have
constant rank while their nonnegative ranks can grow arbitrarily high.

\begin{example}\label{ex:ngon_nonn}
Let $S_n$ be the slack matrix of a regular $n$-gon in the plane. Then
$\rank(S_n)= 3$ for all $n$, while, by Corollary~\ref{cor:bounds on
  nonneg rank examples}, $\rank_+(S_n) \geq \log_2(n)$.
\end{example}

The above lower bound is of optimal order since a regular $n$-gon has a 
$\RR^k_+$-lift where $k = O(\log_2(n))$ by the results in \cite{BenTalNemirovski}. 

The psd rank of a nonnegative matrix or convex body seems to be even
harder to study than nonnegative rank and no techniques are known for
finding upper or lower bounds for it in general. Here we will derive some coarse complexity bounds
by providing bounds for algebraic degrees. To derive our results, we begin with a rephrasing of part of
\cite[Theorem 1.1]{TarskiSeidenberg} about quantifier elimination.


\begin{theorem} 
\label{thm:renegar}
Given a formula of the form 
$$\exists \,\,y \in \RR^{m-n} \,:\, g_i(x,y) \geq 0 \,\,\,\,\forall \,i=1, \ldots, s$$ 
where $x \in \RR^n$ and $g_i \in \RR[x,y]$ are polynomials of degree at most $d$, 
there exists a quantifier elimination method that produces a quantifier free formula of the form 
\begin{equation} \label{quantifier-free formula}
\bigvee_{i=1}^{I} \bigwedge_{j=1}^{J_i} (h_{ij}(x) \,\Delta_{ij} \,0)
\end{equation} 
where $h_{ij} \in \RR[x]$, $\Delta_{ij} \in \{>, \geq, =, \neq, \leq, < \}$ such that 
$$I \leq (sd)^{{K}n(m-n)}, \,\,J_i \leq (sd)^{{K}(m-n)}$$ 
and the degree of  $h_{ij}$ is at most $(sd)^{{K}(m-n)}$, where 
$K$ is a constant.
\end{theorem}

The following result of Renegar on {\em hyperbolic programs} offers a semialgebraic description by 
$k$ polynomial inequalities of degree at most $k$, of an affine 
slice of a $\PSD^k$ (a spectrahedron) that contains a positive definite matrix.

\begin{theorem} \cite{RenegarHP} \label{thm:Renegar derivatives}
Let $Q = \{ z \in \RR^m \,:\, C + \sum z_i A_i \succeq 0 \}$ be a spectrahedron with 
$E := C+\sum z'_i A_i \succ 0$ for some $z' \in Q$, and $C, A_i$ are symmetric matrices of size $k \times k$. Then $Q$ is a semialgebraic set described by $g^{(i)}(z) \geq 0$ for $i=1,\ldots,k$ 
where $g^{(0)}(z) := \textup{det} (C + \sum z_i A_i)$ and $g^{(i)}(z)$ is the $i$-th Renegar derivative of $g^{(0)}(z)$ in direction $E$.
\end{theorem}

With these two results, we can give a lower bound on the psd rank of
a full-dimensional, convex, semi-algebraic set $C$. The {\em Zariski closure} of the 
boundary of $C$ is a hypersurface in $\RR^n$ since the boundary of $C$ has codimension one. We define the {\em degree of $C$} to be the degree of a minimal degree (nonzero) polynomial whose zero set is the 
Zariski closure of the boundary of $C$. By construction, this polynomial vanishes on the boundary of $C$.

\begin{proposition} \label{prop:degree bound}
If $C \subseteq \RR^n$ is a full-dimensional convex semialgebraic set with a $\PSD^k$-lift, 
then the degree of $C$ is at most $k^{O(k^2n)}$.
\end{proposition}

\begin{proof}
We may assume that $C$ has a proper $\PSD^k$-lift since otherwise we
can restrict to a face of $\PSD^k$ and obtain a $\PSD^r$-lift with $r
< k$. Hence there is an affine subspace $L$ that intersects the
interior of $\PSD^k$ such that $C = \pi(\PSD^k \cap L)$.  This implies
that there exist $k \times k$ symmetric matrices $A_1, \ldots, A_n,
B_{n+1}, \ldots, B_m$ and a positive definite matrix $A_0$ such that
$$L = \left\{A_0 + \sum x_i A_i + \sum y_j B_j, \,\,\,\,(x, y) \in \RR^n \times \RR^{m-n} \right\},$$
and $\pi(A_0 + \sum x_i A_i + \sum y_j B_j)=(x_1,\ldots, x_n)$. 
Let $$Q = \left\{ (x,y) \in \RR^n \times \RR^{m-n} \,:\, A_0 + \sum x_i A_i + \sum y_j B_j \succeq 0 \right\}.$$ 
Then 
by Theorem~\ref{thm:Renegar derivatives}, $Q$ is a 
basic semialgebraic set cut out by the $k$ Renegar derivatives, $g_i (x,y) \geq 0$, 
of $\textup{det}(A_0 + \sum x_i A_i + \sum y_j B_j)$, with the degree of each $g_i$ at most $k$. 

Since $C$ is the projection of $Q$, by the Tarski-Seidenberg transfer
principle \cite{MarshallBook}, $C$ is again semialgebraic and has a quantifier
free formula of the type (\ref{quantifier-free formula}).  Hence the
boundary of $C$ is described by at most $(k^2)^{{K}(m-n)(n+1)}$
polynomials of degree at most $(k^2)^{{K}(m-n)}$ where $K$ is a
constant. Since $m < {k+1\choose 2} \leq k^2$, by multiplying all
those polynomials together we get a polynomial vanishing on the
boundary of $C$ of degree at most $(k)^{{2K}(k^2-n)(n+2)}=
k^{O(k^2n)}$.
\end{proof}

The above result provides bounds on the psd ranks of polytopes.

\begin{corollary} \label{cor:max facets with fixed psd rank}
If $C \subset \RR^n$ is a full-dimensional polytope whose slack matrix has psd rank $k$, then $C$ has
at most $k^{O(k^2n)}$ facets.
\end{corollary}
\begin{proof}
If the psd rank of the slack matrix of $C$ is $k$ then $C$ has a
$\PSD^k$-lift. By Proposition~\ref{prop:degree bound} the degree of
$C$ is then at most $k^{O(k^2n)}$. Since the minimal degree polynomial
that vanishes on the boundary of a polytope is the product of the
linear polynomials that vanish on each of its facets, the degree of
$C$ is the number of facets of $C$.
\end{proof}

This shows that even for slack matrices of polytopes there is no function of rank that bounds psd rank.

\begin{example}
As in \ref{ex:ngon_nonn}, let $S_n$ be the slack matrix of a regular
$n$-gon in the plane. Then by Proposition~ \ref{prop:degree bound},
$\rank_{psd}(S_n)$ grows to infinity as $n$ increases. But as we have
seen before, $\rank(S_n)=3$ for all $n$.
\end{example}

In this section, we have shown that the gap between all pairs of ranks: $\rank$, $\rank_+$ and $\rank_{\textup{psd}}$ can become arbitrarily large for nonnegative matrices. For slack matrices of
polytopes we have given examples where the gaps between $\rank$ and
$\rank_+$, and $\rank$  and $\rank_{\textup{psd}}$, can also grow arbitrarily
large. However, no family of slack matrices are known for which
$\rank_+$ can become arbitrarily bigger than $\rank_{\textup{psd}}$ or at least exponentially bigger. 
Such a family would provide the first concrete proof that semidefinite programming can provide smaller 
representations of polytopes than linear programming.


\section{Applications} \label{sec:applications}

\subsection{Stable set polytopes}

An interesting example of polytopes that arise from combinatorial optimization is that of
stable set polytopes. Let $G$ be a graph with vertices
$V=\{1,\ldots,n\}$ and edge set $E$. A subset $S \subseteq
V$ is {\em stable} if there are no edges between elements in $S$. To
each stable set $S$ we can associate a vector $\chi_S \in \{0,1\}^n$
where $(\chi_S)_i = 1$ if $i \in S$ and $(\chi_S)_i = 0$ otherwise.
The {\em stable set polytope} of the graph
$G$ is the polytope
$$\STAB(G)=\conv\{ \chi_S : S \textrm{ is a stable set of } G\}.$$
Finding the largest stable set in a (possibly vertex-weighted) graph is a classic NP-hard problem in combinatorial optimization that can be formulated as linear optimization over $\STAB(G)$.
The polytopes $\STAB(G)$ give rise to one of the most celebrated results in semidefinite lifts of polytopes.
Recall that a graph is {\em perfect} if the chromatic number of every induced subgraph equals the 
size of its largest clique. 

\begin{theorem}\cite{LovaszSchrijver91} \label{thm:lovasz perfect graph}
Let $G$ be a perfect graph with $n$ vertices, then $\STAB(G)$ has a 
$\PSD^{n+1}$-lift.
\end{theorem}

The proof is by explicit construction. Suppose $X \in \PSD^{n+1}$ has rows and columns indexed by $0,1,\ldots,n$. Lov{\'a}sz showed that when $G$ is perfect, the cone $\PSD^{n+1}$ sliced by the planes given by $$X_{0,0}=1,  \,\,\,\,X_{i,i}=X_{0,i} \,\,\,\forall \,\,\, i, \,\,\,\, X_{i,j}=0 \,\,\,\forall \,\,\, (i,j) \in E,$$ and projected onto the coordinates $X_{i,i}$ for $i=1,\ldots,n$, is exactly $\STAB(G)$. If $G$ is not perfect this construction offers a convex relaxation of $\STAB(G)$ called the {\em theta body} of $G$.
In \cite{Yannakakis}, Yannakakis showed that 
if $G$ is perfect, $\STAB(G)$ has a $\RR^k_+$-lift where $k = n^{O(\log \,n)}$. 
It is an open problem as to whether $\STAB(G)$, when $G$ is perfect, admits a polyhedral lift of size polynomial in the number of vertices of $G$. Such a result is plausible since one can find a maximum weight stable set in a perfect graph in polynomial time by semidefinite programming over the above lift. On the other hand, it would also be interesting if $\STAB(G)$ does not 
admit a polyhedral lift of size polynomial in $n$ when $G$ is a perfect graph. Such a result would provide the first example of a family of discrete optimization problems where semidefinite lifts are appreciably smaller than polyhedral lifts. In fact, until recently no explicit family of graphs was known for which $\STAB(G)$ does not admit a polyhedral lift of size polynomial in the number of vertices of $G$. In \cite{FMPTW}, the authors construct non-perfect graphs $G$ with $n$ vertices for which $\rank_+(\STAB(G))$ is $2^{\Omega{(n^{1/2}})}$.

In the context of Theorem~\ref{thm:lovasz perfect graph}, a natural 
question is whether there could exist a positive semidefinite lift of the stable set polytope of a perfect graph to some $\PSD^k$ where $k < n+1$.  The next theorem settles this question.

\begin{theorem} 
\label{thm:theta body is optimal}
Let $G$ be any graph with $n$ vertices. Then $\STAB(G)$ does not admit a $\PSD^n$-lift.
\end{theorem}
\begin{proof}
Using Theorem~\ref{thm:general_Yannakakis_for_polytopes} it is enough
to show that the slack matrix of $\STAB(G)$ has no
$\PSD^n$-factorization. Furthermore we may restrict ourselves to a
submatrix of the slack matrix. Consider the subset $V'$ of vertices of $\STAB(G)$ 
consisting of the origin and all the standard basis vectors $e_1$,\ldots,$e_n$.  The set $V'$ is in the 
vertex set of every stable set polytope since the empty set and all
singleton vertices are stable in any graph. Consider also a set of facets $F'$ 
containing some facet that does not touch the origin, and all $n$ facets given by the 
nonnegativities $x_i \geq 0$. The submatrix of the slack matrix whose
rows are indexed by $V'$ and columns by $F'$ has the block structure
$$S'=\begin{pmatrix}
1 & 0_n \\
*_n &I_n 
\end{pmatrix}$$
where $*_n$ is some unknown $n \times 1$ vector, $0_n$ the zero vector
of size $1 \times n$ and $I_n$ the $n \times n$ identity
matrix. Suppose $S'$ has a $\PSD^n$ factorization with 
$A_0,\ldots,A_{n} \in \PSD^n$ associated to rows and $B_0,\ldots,B_n \in \PSD^n$ 
associated to columns. By looking at the
first row of $S'$  we see $\langle A_0, B_i \rangle= 0$ for all $i \geq 1$ which
implies $A_0B_i=0$ for all $i \geq 1$ since all matrices are psd. Therefore,
the columns of each $B_i$ are in the kernel of
$A_0$ for $i\geq 1$. Since $A_0$ is a nonzero $n\times n$ matrix, its kernel has dimension at most $n-1$, and contains all columns of $B_i$ for $i=1,\ldots,n$. By a dimension count we get that all the columns of one
of the $B_i$, say $B_k$, are in the span of the columns of
$B_i$, $i \geq 1$ and $i \not = k$.  Consider now $A_k$. Again, $A_k B_i = 0$ for all $i\geq
1$ and $i \not = k$, which implies that all columns of those $B_i$ are
in the kernel of $A_k$. But this implies that so are the columns of
$B_k$. Therefore, $\langle B_k,A_k \rangle= 0$ which contradicts the structure of $S'$.
\end{proof}

\begin{remark}
\begin{enumerate}
\item In fact, the above proof shows that any polytope in $\RR^n$ that 
 has a vertex that locally looks like a nonnegative
orthant has no $\PSD^n$-lift. Recently, it has been shown \cite{GRT} that 
the psd rank of a $n$-dimensional polytope in $\RR^n$ is at least $n+1$.
\item The result in Theorem~\ref{thm:theta body is optimal} is simple, and yet remarkable in a couple of ways. First, it is an illustration of the usefulness of the factorization theorem (Theorem~\ref{thm:general_Yannakakis}) to prove the optimality of a lift. Secondly,  it is impressive that the simple and natural semidefinite lift proposed by Lov{\'a}sz is optimal in this sense. 
\item Theorem~4.2 in \cite{GPT1} implies that any $n$-dimensional polytope with a $0/1$-slack matrix admits a $\PSD^{n+1}$-lift. A simple proof of this fact follows from Proposition~\ref{prop:squaring} since the rank of a slack matrix of a polytope in $\RR^n$ is at most $n+1$.
\end{enumerate}
\end{remark}

We close this subsection with an interesting class of lifts of stable set
polytopes to completely positive cones. Recall that $\C_{n}^*$ is the cone of $n \times n$ completely positive matrices. 

\begin{theorem} \cite{Burer} \label{thm:Burer}
For any graph $G$ with $n$ vertices, the polytope $\STAB(G)$ has a $\C_{n+1}^*$-lift.
\end{theorem}
\begin{proof}
This is an immediate consequence of Proposition 3.2 in \cite{Burer}
applied to this problem.
\end{proof}

The $\C_{n+1}^*$-lift of $\STAB(G)$  is given by the same linear constraints on  $X \in \C_{n+1}^*$ that were used to construct
the $\PSD^{n+1}$-lift.
These lifts are of very small size and work for all graphs, but have
limited interest in practical computations since copositive/completely
positive programming is not known to have any efficient algorithms. We illustrate the copositive/completely positive factorization that is expected for this lift in the case of a $5$-cycle. 

\begin{example} \label{ex:Burer}
From Theorem~\ref{thm:Burer} we know that the stable set polytope of a $5$-cycle has a $C_{6}^*$-lift, and hence by Theorem~\ref{thm:general_Yannakakis}, its slack matrix must
have a $C_{6}^*$-factorization. This polytope has $11$ vertices:
the origin, the five standard basis vectors $e_1,\ldots,e_5$ and the five sums $e_1+e_3, e_2+e_4, e_3+e_5, e_4+e_1, e_5+e_2$ corresponding to the five
stable sets of the $5$-cycle with two elements. We will denote these last five vertices by $s_1, \ldots, s_5$, respectively.
Furthermore, there are $11$ facets for this stable set polytope given by the inequalities:
$$x_i \geq 0, \,\,\,\, x_i + x_{i+1} \leq 1,  \,\,\,\forall \,\,\,i=1,\ldots,5, \,\,\mbox{ and } \,\, \sum_{j=1}^5 x_j \leq 2$$ 
where we identify $x_6$ with $x_1$.

Since we know a $\C_6^*$-lift, the $A$ map that takes vertices of $\STAB(G)$ to $\C_{6}^*$ is easy to get.
Send each vertex $v \in \RR^5$ to $A(v)=(1,v)^T (1,v) \in \RR^{6 \times 6}$, and since all
coordinates are nonnegative, $A(v)$ is completely positive. For the copositive lifts of the facets, we go case by case. For $x_i\geq 0$
take the matrix $(0,e_i)^T (0,e_i)$, while for $1-x_i-x_{i+1} \geq 0$ take $(1,-e_i-e_{i+1})^T (1,-e_i-e_{i+1})$. All these matrices are positive semidefinite and hence also copositive. It is also easy to check that they satisfy 
the factorization requirements.

It remains to find a copositive matrix for the odd-cycle inequality $2-\sum_{j=1}^5 x_j \geq 0$. This is non-trivial, but it can be checked that the following matrix works for the factorization:
$$\left(\begin{array}{cccccc}
2  & -1& -1& -1& -1& -1 \\
-1 & 1 &  1 & 0 & 0 & 1 \\
-1 & 1 &  1 &1 & 0 & 0 \\
-1 & 0&  1 & 1 & 1 & 0 \\
-1 & 0 &  0 & 1 & 1 & 1 \\
-1 & 1 &  0 & 0 & 1 & 1 
\end{array}\right).$$
To see that it is copositive, by Theorem 2 in \cite{Burer}, we just have to show that
$$2 \left(\begin{array}{ccccc}
 1 &  1 & 0 & 0 & 1 \\
 1 &  1 &1 & 0 & 0 \\
 0&  1 & 1 & 1 & 0 \\
 0 &  0 & 1 & 1 & 1 \\
 1 &  0 & 0 & 1 & 1 
\end{array}\right) 
-\left(\begin{array}{c}
 -1\\ -1\\ -1\\ -1\\ -1 
\end{array}\right)\left(\begin{array}{c}
 -1\\ -1\\ -1\\ -1\\ -1 
\end{array}\right)^T 
= \left(\begin{array}{ccccc}
 1 &  1 & -1 & -1 & 1 \\
 1 &  1 &1 & -1 & -1 \\
 -1&  1 & 1 & 1 & -1 \\
 -1 &  -1 & 1 & 1 & 1 \\
 1 &  -1 & -1 & 1 & 1 
\end{array}\right) $$
is copositive, and this is a well known Horn form, that is copositive. For a proof, see for instance, 
Lemma 2.1 in \cite{Baston}.
\end{example}

\medskip

\subsection{Rational lifts of algebraic sets}

Our last application is an interpretation of Theorem~\ref{thm:general_Yannakakis} for an important class of positive semidefinite lifts, called {\em rational lifts}, of zero sets of polynomial equations. Suppose we have a system of polynomial equations 
\begin{equation} \label{eqn:poly system}
p_1(x)=p_2(x)=\cdots=p_m(x)=0
\end{equation}
where the $p_i$'s have real coefficients and $n$ variables, and 
$I$ is the {\em ideal} they generate in the polynomial ring $\RR[x] = \RR[x_1,\ldots,x_n]$. The set of zeros of 
(\ref{eqn:poly system}), denoted by $\V_{\RR}(I)$, is the {\em real variety} of the ideal $I$, and we consider positive semidefinite lifts
of $C=\conv(\V_{\RR}(I))$. Since different polynomial systems can generate the same convex hull, we define the {\em convex radical 
ideal} of $I$ to be the ideal $\sqrt[\conv]{I}$ of polynomials vanishing on $\V_{\RR}(I) \cap \ext(C)$. Replacing $I$ by $\sqrt[\conv]{I}$ does not change $C$ and so we will, to simplify arguments, assume that $I = \sqrt[\conv]{I}$.

We consider special kinds of $\PSD^k$-factorizations of the slack operator $S_C$, namely, those where the map $A:\ext(C)\rightarrow \PSD^k$ is of the form $A(x)=v(x)v(x)^T$, where $v(x)$ is a vector of rational functions $v(x) = (v_1(x),\ldots,v_n(x))$. By factoring out the common denominators, 
we can rewrite such a map
as $A(x)= \frac{1}{p(x)^2} w(x)w(x)^T$ where $w(x)$ is a vector of polynomials.
We say that $A$ is a {\em rational map}, and if $p(x)=1$ 
we say that $A$ is a {\em polynomial map}. A $\PSD^k$-factorization of $S_C$ is called a rational (respectively, polynomial) factorization
if the map $A$ used in the factorization is a rational (respectively, polynomial) map.

These lifts turn out to be related to the sums of squares techniques for lift-and-project methods. Given a polynomial $q(x) \in \RR[x]$, we say that it is
a {\em sum of squares (sos) modulo} $I$, if there exist polynomials $h_1(x), \ldots, h_s(x) \in \RR[x]$ 
such that $q(x) - \sum h_i(x)^2 \in I$. If the degrees of all the $h_i$ are bounded above by $k$ we say that $q$ is $k$-{\em sos modulo} $I$. 
This is a sufficient
condition for nonnegativity over a real variety that has been used to construct sequences of semidefinite relaxations of the convex hull of the variety. One such hierarchy is given by the {\em theta bodies} of $I$, introduced in 
\cite{GPT1}.  They are defined geometrically by taking the $k$-th
theta body relaxation of $\conv(\V_\RR(I))$, denoted as $\TH_k(I)$, to
be the intersection of all half-spaces $\{x \,:\, \ell(x) \geq 0\}$ where $\ell(x)$ is a linear polynomial that is $k$-sos modulo $I$.

\begin{theorem} \label{thm:moments}
Let $I$ be a convex radical ideal and $Z=\V_{\RR}(I)$ its zero set such that $\conv(Z)$ is compact and contains the origin. Then,
\begin{enumerate}
\item the slack operator of $\conv(Z)$ has a rational factorization with $A(x)=\frac{1}{p(x)^2} w(x)w(x)^T$ in 
$\PSD^k$ for all $x \in \ext(\conv(Z))$ if 
and only if, for every linear polynomial $\ell(x)$
 nonnegative over $Z$, $p(x)^2 \ell(x)$ is a sum of squares modulo $I$, with all the polynomials in the sum of squares being linear combinations
 of the entries of $w(x).$
\item The slack operator of $\conv(Z)$ has a polynomial factorization with $A(x)= w(x)w(x)^T$ where the degree of each entry in $w$ at most $k$ if and only if $\TH_k(I)=\conv(Z)$.
\end{enumerate}
\end{theorem}
\begin{proof}
For the first part note that since any linear polynomial $\ell(x)$
 nonnegative over $Z$ is a convex combination of extreme points of the polar of $\conv(Z)$, there exists a matrix $B_{\ell} \in \PSD^k$ such that
 $\ell(x)= \left<B_{\ell},A(x)\right>$ for all $x \in \ext(\conv(Z)))$. Since $I$ is convex radical this actually implies
 $\ell(x)= \left<B_{\ell},A(x)\right>$ modulo $I$, and by rewriting the right hand side we have
 $p(x)^2\ell(x)= w(x)^T B_{\ell} w(x)$ modulo $I$, which is a sum of squares modulo the ideal with the conditions we want. Since all steps in the
 proof are actually equivalences, this gives us a proof of the first statement.
 
 For the second statement just note that from \cite{GPT1}, $I$ is $\TH_k$-exact if and only if all linear polynomials non-negative over $\V_{\RR}(I)$
 are $k$-sos modulo $I$ (since $I$ is in particular real radical). Now use the first statement to conclude the proof.
\end{proof}

A rational factorization of the slack matrix of $C := \conv(\V_\RR(I))$ consists of two maps $A$ and $B$ that assign psd matrices to extreme points of $C$ and $C^\circ$. On the {\em primal} side, every extreme point (and hence every point) of $C$ is being lifted to a psd matrix  via the map $A$. On the {\em dual} side, $B$ is assigning a psd Gram matrix to every linear functional that is nonnegative on $\V_\RR(I)$ certifying its sum of squares property with respect to this variety.

Several further remarks are in order.  
The requirements that $\conv(Z)$ is compact and contains the origin in its interior are not essential and are
assumed for the sake of simplicity and to keep the discussion in the same setting as in our main theorems.  A similar idea could be applied to convex hulls of sets defined by polynomial inequalities, but there the usual lift is not to a positive semidefinite cone but to a product of such cones, making the notation more cumbersome. Finally, the condition that the ideal $I$ is convex
radical can be avoided if we use a stronger notion of a polynomial lift that implies factorization over the entire variety and not just over the extreme points of the convex hull of the variety.

\bibliographystyle{plain}
\bibliography{all}

\end{document}